\documentclass[11pt]{article}
\usepackage{amsmath,amsthm,amssymb}
\usepackage[dvips]{graphicx}

\newcommand{\R}{\mathbb R}
\newcommand{\Z}{\mathbb Z}
\newcommand{\ii}{{\operatorname{i}}}
\newcommand{\dd}{{\operatorname{d}}}
\newcommand{\sn}{{\operatorname{sn}}}
\newcommand{\cn}{{\operatorname{cn}}}
\newcommand{\dn}{{\operatorname{dn}}}
\newcommand{\zn}{{\operatorname{zn}}}

\newcommand{\laK}{\lambda{K}}
\newcommand{\muK}{\mu{K}}

\begin{document}


\title{Some New Properties of Jacobi's Theta Functions\footnote{published in: Journal of Computational and Applied Mathematics {\bf 178} (2005), 419--424.}}
\author{Klaus Schiefermayr\footnote{University of Applied Sciences Upper Austria, School of Engineering and Environmental Sciences, Stelzhamerstrasse\,23, 4600 Wels, Austria, \textsc{klaus.schiefermayr@fh-wels.at}}}
\date{}
\maketitle

\theoremstyle{plain}
\newtheorem{theorem}{Theorem}
\newtheorem{lemma}[theorem]{Lemma}
\theoremstyle{definition}
\newtheorem*{remark}{Remark}

\thispagestyle{empty}

\begin{abstract}
In this paper, a monotonicity property for the quotient of two Jacobi's theta functions with respect to the modulus $k$ is proved.
\end{abstract}

\noindent\emph{Mathematics Subject Classification (2000):} 33E05

\noindent\emph{Keywords:} Jacobi's elliptic functions, Jacobi's theta functions

\section{Introduction and Main Result}


Let $k\in(0,1)$ denote the modulus of Jacobi's elliptic functions $\sn(u)$, $\cn(u)$, and $\dn(u)$, of Jacobi's theta functions $\Theta(u)$, $H(u)$, $H_1(u)$, and $\Theta_1(u)$, and, finally, of Jacobi's zeta function, $\zn(u)$. Here we follow the notation of Carlson and Todd\,\cite{CarlsonTodd}, in other references, like \cite{WhittakerWatson}, Jacobi's zeta function is denoted by $Z(u)$. Let $k':=\sqrt{1-k^2}\in(0,1)$ be the complementary modulus, let $K\equiv K(k)$ and $E\equiv E(k)$ be the complete elliptic integral of the first and second kind, respectively, and let $K'\equiv K'(k):=K(k')$ and $E'\equiv E'(k):=E(k')$.


If we want to point out the dependence of these functions on the modulus $k$, we will write $K(k)$, $K'(k)$, $\sn(u,k)$, $\Theta(u,k)$, etc. In this paper, we follow the old notation of the theta functions, which goes back to Jacobi. There is another notation of the four theta functions given by $\Theta(u,k)=\vartheta_0(v,\tau)=\vartheta_4(v,\tau)$, $H(u,k)=\vartheta_1(v,\tau)$, $H_1(u,k)=\vartheta_2(v,\tau)$ and $\Theta_1(u,k)=\vartheta_3(v,\tau)$, where $\tau=\ii K'/K$ and $v=u/(2K)$ (note that in some references, like \cite{ByrdFriedman} and \cite{Lawden}, $v=u\pi/(2K)$).


For the definitions and many important properties of these functions, see, e.g., \cite{ByrdFriedman}, \cite{Lawden}, \cite{MagnusOberhettingerSoni}, \cite{TanneryMolk}, and \cite{Achieser}.


In \cite{CarlsonTodd}, Carlson and Todd proved that, for each $\lambda\in(0,1)$, the functions $\sn(\laK,k)$ and $\zn(\laK,k)$ are strictly monotone increasing with respect to the modulus $k$, $0<k<1$. In addition, they investigated the degenerating behaviour of these functions as $k\to0$ and especially as $k\to1$. Hence the question arises, whether analogous monotonicity properties hold for the theta functions. Unfortunately, $\Theta(\laK,k)$ does not have the same monotonicity behaviour with respect to the modulus $k$ for every $\lambda\in(0,1)$. Numerical examples show that for small $\lambda$ ($\lambda\leq0.5$), $\Theta(\laK)$ is strictly monotone decreasing, for large $\lambda$ ($\lambda\geq0.6$), $\Theta(\laK)$ is strictly monotone increasing in $k$ and for some $\lambda\in(0.5,0.6)$, $\Theta(\laK)$ is not monotone at all in the whole interval $(0,1)$. However, we are able to prove that for each $\lambda,\mu\in\R$, the quotient $\Theta(\laK)/\Theta(\muK)$ of two theta functions is strictly monotone. In addition, the degenerative bevaviour of $\Theta(\laK)/\Theta(\muK)$ as $k\to0$ and $k\to1$ is given.


\begin{theorem}\label{Theorem_1}\hfill{}
\begin{enumerate}
\item Let $\lambda,\mu\in\R$. If $\cos(\lambda\pi)>\cos(\mu\pi)$ $\bigl[\cos(\lambda\pi)<\cos(\mu\pi)\bigr]$, then $\Theta(\laK)/\Theta(\muK)$ is a positive, strictly monotone decreasing $\bigl[$increasing$\bigr]$ function of the modulus $k\in(0,1)$. If $\cos(\lambda\pi)=\cos(\mu\pi)$, i.e., $\lambda=\mu+2\nu$, $\nu\in\Z$, then $\Theta(\laK)/\Theta(\muK)=1$.
\item Let $\lambda,\mu\in(0,1)$, $k\in(0,1)$, then $\Theta(\laK)/\Theta(\muK)\to1$ as $k\to0$ and
\[
\frac{\Theta(\laK)}{\Theta(\muK)}\sim\Bigl(\frac{k'}{4}\Bigr)^{(\mu-\lambda)(1-(\lambda+\mu)/2)}~\text{as}~k\to1.
\]
\item Let $\mu\in(0,1)$ and $k\in(0,1)$. Then $f(\lambda):=\Theta((\mu-\lambda)K)/\Theta((\mu+\lambda)K)$ is a convex function of $\lambda\in(0,1)$ with $f(0)=f(1)=1$.
\end{enumerate}
\end{theorem}


\begin{remark}
By the relation $\Theta_1(u)=\Theta(u+K)$, one gets analogous monotonicity properties for the quotients $\Theta_1(\laK)/\Theta_1(\muK)$, $\Theta(\laK)/\Theta_1(\muK)$, and $\Theta_1(\laK)/\Theta(\muK)$.
\end{remark}


In \cite{PehSch}, we considered polynomials, whose $[-1,1]$ inverse image consists of two Jordan arcs, i.e., we characterized polynomials $P_n$, for which $P_n^{-1}([-1,1])$ consists of two Jordan arcs (in general, $P_n^{-1}([-1,1])$ consists of $n$ Jordan arcs). Since these polynomials $P_n$ can be given with the help of an elliptic integral, Jacobi's elliptic and theta functions appear in a natural way. When describing the shape of the two Jordan arcs, we need the above theorem, see Theorem\,22, Lemma\,33 and Lemma\,34 of \cite{PehSch}.

\section{Proof of the Main Result}


First, we collect some derivation formulas, which are an immediate consequence of formula (710.00) of \cite{ByrdFriedman}.


\begin{lemma}\label{Lemma1}
Let $k\in(0,1)$. Then
\begin{gather*}
\frac{\dd K}{\dd{k}}=\frac{E-{k'}^2K}{k{k'}^2}, \quad
\frac{\dd K'}{\dd{k}}=\frac{k^2K'-E'}{k{k'}^2}, \\
\frac{\dd}{\dd{k}}\Bigl\{\frac{1}{K}\Bigr\}=\frac{{k'}^2K-E}{k{k'}^2K^2},
\quad\frac{\dd}{\dd{k}}\Bigl\{\frac{K'}{K}\Bigr\}=\frac{-\pi}{2k{k'}^2K^2}
\end{gather*}
\end{lemma}


Concerning the limiting behaviour of $\sn(\laK)$, etc., as $k\to1$, Carlson and Todd\,\cite{CarlsonTodd} have proved the following.


\begin{lemma} \label{Lemma_LimitingBehaviour}
Let $0<\lambda<1$, then, as $k\to1$, we have
\begin{gather*}
K\sim\log(\tfrac4{k'}),\quad\sn(\laK)\sim1-2(\tfrac{k'}4)^{2\lambda},\quad\cn(\laK)\sim\dn(\laK)\sim2(\tfrac{k'}4)^\lambda,\\
\zn(\laK)\sim 1-\lambda-2(\tfrac{k'}4)^{2\lambda}.
\end{gather*}
\end{lemma}


Moreover, we will need the following formulas for the derivatives of the theta functions, which are a direct consequence of (1053.01), (1052.02) and (731.01)--(731.03) of \cite{ByrdFriedman}.


\begin{lemma}\label{Lemma_DerivTheta}
The following relations hold:
\[
\begin{aligned}
\frac{\partial}{\partial u}\bigl\{\Theta(u)\bigr\}&=\Theta(u)\zn(u),\\
\frac{\partial}{\partial u}\bigl\{H(u)\bigr\}&=\sqrt{k}\,\Theta(u)\bigl(\cn(u)\dn(u)+\sn(u)\zn(u)\bigr),\\
\frac{\partial}{\partial u}\bigl\{H_1(u)\bigr\}&=\tfrac{\sqrt{k}}{\sqrt{k'}}\,\Theta(u)\bigl(-\sn(u)\dn(u)+\cn(u)\zn(u)\bigr),\\
\frac{\partial}{\partial u}\bigl\{\Theta_1(u)\bigr\}&=\tfrac{1}{\sqrt{k'}}\,\Theta(u)\bigl(-k^2\sn(u)\cn(u)+\dn(u)\zn(u)\bigr),
\end{aligned}
\]
and
\[
\begin{aligned}
\frac{\partial^2}{\partial u^2}\bigl\{\Theta(u)\bigr\}&=\Theta(u)\bigl(\dn^2(u)+\zn^2(u)-E/K\bigr),\\
\frac{\partial^2}{\partial u^2}\bigl\{H(u)\bigr\}&=\sqrt{k}\,\Theta(u)\Bigl(-k^2\sn(u)\cn^2(u)+2\cn(u)\dn(u)\zn(u)\\
&\qquad+\sn(u)\bigl(\zn^2(u)-E/K\bigr)\Bigr),\\
\frac{\partial^2}{\partial u^2}\bigl\{H_1(u)\bigr\}&=\tfrac{\sqrt{k}}{\sqrt{k'}}\,\Theta(u)
\Bigl(-k^2\sn^2(u)\cn(u)-2\sn(u)\dn(u)\zn(u)\\
&\qquad+\cn(u)\bigl(\zn^2(u)-E/K\bigr)\Bigr),\\
\frac{\partial^2}{\partial u^2}\bigl\{\Theta_1(u)\bigr\}&=\tfrac{1}{\sqrt{k'}}\,\Theta(u)
\Bigl(\dn(u)\bigl(1-k^2\cn^2(u)\bigr)-2k^2\sn(u)\cn(u)\zn(u)\\
&\qquad+\dn(u)\bigl(\zn^2(u)-E/K\bigr)\Bigr).
\end{aligned}
\]
\end{lemma}


\begin{remark}
By (123.01), (123.03) of \cite{ByrdFriedman} and Lemma\,\ref{Lemma_DerivTheta}, the relation
\[
\frac{\partial}{\partial{u}}\Bigl\{\log\Bigl(\frac{H(v-u)}{H(v+u)}\Bigr)\Bigr\}
=\frac{2\,\sn(v)\,\cn(v)\,\dn(v)}{\sn^2(u)-\sn^2(v)}-2\,\zn(v)
\]
holds which implies
\begin{equation} \label{Log(H/H)}
\log\Bigl(\frac{H(v-u)}{H(v+u)}\Bigr)=\int_0^u\frac{2\,\sn(v)\,\cn(v)\,\dn(v)}{\sn^2(u)-\sn^2(v)}\,\dd{u}-2\,u\,\zn(v).
\end{equation}
Analogous formulas can be obtained for other ratios of theta functions.
\end{remark}


The next lemma gives the derivatives with respect to the modulus $k$ for the four theta functions $\Theta(u)$, $H(u)$, $H_1(u)$, $\Theta_1(u)$, where $u=\laK$. Note that $\laK$ as well as the theta functions themselves depend on the modulus $k$.


\begin{lemma}\label{Lemma_DerivTheta_k}
For $\lambda\in\R$, the derivatives with respect to $k$ of the four theta functions $\Theta(\laK)$, $H(\laK)$, $H_1(\laK)$, $\Theta_1(\laK)$ are given by
\[
\begin{aligned}
\frac{\dd}{\dd k}\bigl\{\Theta(\laK)\bigr\}&=-\frac{1}{2k{k'}^2}\,\frac{\partial^2}{\partial{u}^2}
\bigl\{\Theta(u)\bigr\}\Bigl|_{u=\laK}\\
&=-\frac{1}{2k{k'}^2}\,\Theta(\laK)\bigl(\dn^2(\laK)+\zn^2(\laK)-E/K\bigr),\\
\frac{\dd}{\dd{k}}\bigl\{H(\laK)\bigr\}&=-\frac{1}{2k{k'}^2}\,\frac{\partial^2}{\partial{u}^2}
\bigl\{H(u)\bigr\}\Bigl|_{u=\laK}\\
\frac{\dd}{\dd k}\bigl\{H_1(\laK)\bigr\}&=-\frac{1}{2k{k'}^2}\,\frac{\partial^2}{\partial{u}^2}
\bigl\{H_1(u)\bigr\}\Bigl|_{u=\laK}\\
\frac{\dd}{\dd k}\bigl\{\Theta_1(\laK)\bigr\}&=-\frac{1}{2k{k'}^2}\,\frac{\partial^2}{\partial{u}^2}
\bigl\{\Theta_1(u)\bigr\}\Bigl|_{u=\laK}\\
\end{aligned}
\]
\end{lemma}


\begin{proof}
The four theta functions $\vartheta_j(v,\tau)$, $j=1,2,3,4$, satisfy a differential equation of the form
\[
\frac{\partial^2}{\partial v^2}\bigl\{\vartheta_j(v,\tau)\bigr\}
=4\ii\pi\,\frac{\partial}{\partial\tau}\bigl\{\vartheta_j(v,\tau)\bigr\},
\]
see \cite[p.\ 375]{MagnusOberhettingerSoni}. By Lemma\,\ref{Lemma1} and since $\Theta(u,k)=\vartheta_4(\tfrac{u}{2K},\tau)$, where $\tau=\ii{K'}/K$,
\begin{align*}
\frac{\partial}{\partial k}\bigl\{\Theta(u,k)\bigr\}
&=\frac{\partial}{\partial{v}}\bigl\{\vartheta_4(\tfrac{u}{2K},\tau)\bigr\}\,
\frac{\dd}{\dd k}\bigl\{\frac{u}{2K}\bigr\}
+\frac{\partial}{\partial\tau}\bigl\{\vartheta_4(\tfrac{u}{2K},\tau)\bigr\}\,
\frac{\dd\tau}{\dd k} \\
&=\frac{u({k'}^2K-E)}{k{k'}^2K}\,\frac{\partial}{\partial{u}}\bigl\{\Theta(u,k)\bigr\}
-\frac{1}{2k{k'}^2}\,\frac{\partial^2}{\partial{u}^2}\bigl\{\Theta(u,k)\bigr\}.
\end{align*}
Thus, for the derivative with respect to $k$, by Lemma\,\ref{Lemma1}, we get
\begin{align*}
&\frac{\dd}{\dd{k}}\bigl\{\Theta(\laK,k)\bigr\}
=\frac{\dd}{\dd{k}}\bigl\{\laK\bigr\}\,\frac{\partial}{\partial{u}}
\bigl\{\Theta(u,k)\bigr\}\Bigl|_{u=\laK}
+\frac{\partial}{\partial k}\bigl\{\Theta(u,k)\bigr\}\Bigl|_{u=\laK} \\
&=\frac{\lambda(E-{k'}^2K)}{k{k'}^2}\,\frac{\partial}{\partial{u}}
\bigl\{\Theta(u,k)\bigr\}\Bigl|_{u=\laK}
+\frac{\laK({k'}^2K-E)}{k{k'}^2K}\,\frac{\partial}{\partial{u}}
\bigl\{\Theta(u,k)\bigr\}\Bigl|_{u=\laK} \\
&\quad-\frac{1}{2k{k'}^2}\,\frac{\partial^2}{\partial{u}^2}
\bigl\{\Theta(u,k)\bigr\}\Bigl|_{u=\laK} \\
&=-\frac{1}{2k{k'}^2}\,\frac{\partial^2}{\partial{u}^2}
\bigl\{\Theta(u,k)\bigr\}\Bigl|_{u=\laK}.
\end{align*}
Since the last identity holds for the other three theta functions as well, this gives the assertion.
\end{proof}


\begin{proof}[\bf Proof of Theorem 1]\hfill{}
\begin{enumerate}
\item By Lemma\,\ref{Lemma_DerivTheta_k},
\begin{align*}
\frac{\dd}{\dd{k}}\Bigl\{\frac{\Theta(\laK)}{\Theta(\muK)}\Bigr\}
&=-\frac1{2k{k'}^2}\frac{\Theta(\laK)}{\Theta(\muK)} \\
&\quad\times\bigl(\dn^2(\laK)+\zn^2(\laK)-\dn^2(\muK)-\zn^2(\muK)\bigr).
\end{align*}
Thus, it remains to be shown that $g(u):=\dn^2(u)+\zn^2(u)$ satisfies the inequality $g(\laK)-g(\muK)>0$ $[<0]$ for all $\lambda,\mu\in\R$ with $\cos(\lambda\pi)>[<]\cos(\mu\pi)$. This property holds since $g(-u)=g(u)$, $g(u+2K)=g(u)$ and $g(u)$ is a positive, strictly monotone decreasing function in $(0,K)$. Concerning the monotonicity of $g(u)$, note that
\[
g'(u)=-2\dn^2(u)\Bigl(\underbrace{\frac{k^2\sn(u)\cn(u)}{\dn(u)}-\zn(u)}_{=:h(u)}\Bigr)-\frac{2\,E\,\zn(u)}{K},
\]
where $h(0)=h(K)=0$ and
\[
h''(u)=-\frac{2k^2{k'}^2\sn(u)\cn(u)}{\dn^3(u)}<0 \quad\text{for}~u\in(0,K).
\]
Thus $h(u)>0$ and therefore $g'(u)<0$ for $u\in(0,K)$.
\item Assume that $\lambda<\mu$ and let $\alpha:=(\lambda+\mu)/2$, $\beta:=(\mu-\lambda)/2$, i.e.\ $\lambda=\alpha-\beta$, $\mu=\alpha+\beta$, $\alpha,\beta\in(0,1)$. By (1052.02) of \cite{ByrdFriedman} and \eqref{Log(H/H)}, we get
\begin{gather*}
\log\Bigl(\frac{\Theta(\laK)}{\Theta(\muK)}\Bigr)=\log\Bigl(\frac{\sn(\muK)H(\laK)}{\sn(\laK)H(\muK)}\Bigr)=\log(\sn(\muK))-\log(\sn(\laK))\\
+2K\sn(\alpha K)\cn(\alpha K)\dn(\alpha{K})\int_{0}^{\beta}\frac{\dd\nu}{\sn^2(\nu K)-\sn^2(\alpha{K})}-2\beta{K}\zn(\alpha{K}).
\end{gather*}
Thus, as $k\to0$,
\begin{gather*}
\log\Bigl(\frac{\Theta(\laK)}{\Theta(\muK)}\Bigr)\sim \log(\sin(\tfrac{\mu\pi}{2}))-\log(\sin(\tfrac{\lambda\pi}{2}))\\
+\pi\sin(\tfrac{\alpha\pi}{2})\cos(\tfrac{\alpha\pi}{2})
\int_{0}^{\beta}\frac{\dd\nu}{\sin^2(\frac{\nu\pi}{2})-\sin^2(\frac{\alpha\pi}{2})}=0,
\end{gather*}
and, as $k\to1$, by Lemma\,\ref{Lemma_LimitingBehaviour},
$a:=k'/4$,
\begin{align*}
\log&\Bigl(\frac{\Theta(\laK)}{\Theta(\muK)}\Bigr)\sim \log(1-2a^{2\mu})-\log(1-2a^{2\lambda})\\
&-2\log(a)(1-2a^{2\alpha})4a^{2\alpha}\int_{0}^{\beta}\frac{\dd\nu}{4(a^{2\alpha}-a^{2\nu})}+2\beta(1-\alpha-2a^{2\alpha})\log(a)\\
&\sim 2\log(a)\int_{0}^{\beta}\frac{\dd\nu}{a^{2(\nu-\alpha)}-1}+2\beta(1-\alpha)\log(a)\\
&=-2\beta\log(a)+\log(a^{2(\beta-\alpha)}-1)-\log(a^{-2\alpha}-1)+2\beta(1-\alpha)\log(a)\\
&\sim -2\beta\log(a)+2(\beta-\alpha)\log(a)+2\alpha\log(a)+2\beta(1-\alpha)\log(a)\\
&=2\beta(1-\alpha)\log(a).
\end{align*}
If $\lambda>\mu$, note that by the above proved result
\begin{gather*}
\log\frac{\Theta(\laK)}{\Theta(\muK)} =-\log\frac{\Theta(\muK)}{\Theta(\laK)}
\sim-(\lambda-\mu)(1-(\lambda+\mu)/2)\log{a}\\
=(\mu-\lambda)(1-(\lambda+\mu)/2)\log{a}.
\end{gather*}
\item Obviously, $f(0)=1$ and, by (1051.02) and (1051.03) of \cite{ByrdFriedman},
\[
f(1)=\frac{\Theta(\muK-K)}{\Theta(\muK+K)}=\frac{\Theta_1(-\muK)}{\Theta_1(\muK)}=1.
\]
Further, by Lemma\,\ref{Lemma_DerivTheta},
\begin{align*}
f''(\lambda)=K^2f(\lambda)\Bigl(\bigl(&\zn((\mu-\lambda)K)+\zn((\mu+\lambda)K)\bigr)^2\\
&+\dn^2((\mu-\lambda)K)-\dn^2((\mu+\lambda)K)+\frac{2E}{K^2}\Bigr)>0
\end{align*}
for every $\lambda\in(0,1)$, which gives the assertion.
\end{enumerate}
\end{proof}


\bibliographystyle{amsplain}

\bibliography{ThetaFunctions}

\end{document}